\documentclass{article}
\usepackage[utf8]{inputenc}
\usepackage{amsthm,amssymb,geometry}

\usepackage{tikz,subcaption}
\usepackage[colorinlistoftodos]{todonotes}

\title{Borel Polychromatic Number of Grids}
\author{Katalin Berlow and Edward Hou}
\date{\today}

\usepackage{amsmath,amsthm,amssymb,amsbsy,graphicx,comment,thmtools}

\usepackage{thmtools}
\usepackage{hyperref}
\usepackage[capitalize]{cleveref}\usepackage{thm-restate}

\newcommand{\Z}{\mathbb{Z}}

\newcommand{\N}{\mathbb{N}}

\newcommand{\actson}{\curvearrowright}

\newtheorem{theorem}{Theorem}[section]
\newtheorem{lemma}[theorem]{Lemma}

\newtheorem{cor}[theorem]{Corollary}

\newtheorem{ques}[theorem]{Open Question}
\theoremstyle{definition}
\newtheorem{definition}[theorem]{Definition}
\newcounter{xmpl}
\newtheorem{example}[xmpl]{Example}

\begin{document}

\maketitle

\begin{abstract}
We study Borel polychromatic colorings of grid graphs arising from free Borel actions of $\mathbb{Z}^d$. 
A polychromatic coloring is one in which every unit $d$-dimensional cube sees all available colors. In the classical setting, every grid admits a $2^d$-polychromatic coloring, while in the Borel setting this fails. Our main result shows that every free $\Z^d$-action admits a Borel $(2^d-1)$-polychromatic coloring. This result is sharp: any action where the generators act ergodically does not admit a Borel $2^d$-polychromatic coloring.  
We conclude with open directions for extending the theory beyond cube tilings and for exploring the dependence of Borel polychromatic numbers on the underlying action.
\end{abstract}

\section{Introduction} 
A labeling of a graph together with a finite collection of sets is called \emph{polychromatic} if every member of a specified family of finite sets sees all colors. The first example of this to be studied was \emph{face-polychromatic} colorings of planar graphs embedded in the plane \cite{Alon-poly}. These are colorings in which every face of $G$ must contain each color. More generally, if $G$ is a graph and $S$ is a family of subsets of $V(G)$, a $k$-labeling is $k$-polychromatic with respect to $S$ if every $s \in S$ contains vertices of all $k$ colors. The \textit{largest} such $k$ is called the \emph{polychromatic number} $\chi^p(G,S)$. These problems have been studied extensively \cite{Polycite1, Polycite2} in finite combinatorics, where they connect to classical questions in hypergraph coloring, set systems, and additive combinatorics.

In this paper we initiate the study of the Borel analog of polychromatic coloring problems. Our focus is on grid graphs arising from free Borel actions of $\mathbb{Z}^d$. Given such an action $\mathbb{Z}^d \actson X$, the associated \emph{Schreier graph} $G$ has vertex set $X$ and edges connecting $x \in X$ to $e_i \cdot x$ for each of the generators $e_0, \dots, e_{d-1}$ of $\mathbb{Z}^d$. We call the Schreier graph of a free action of $\Z^d$ a \textit{grid graph}. We take the family of subsets $S$ to be the collection of $d$-dimensional unit hypercubes in $G$, namely sets of the form $\{0,1\}^d\cdot x$ for $x\in X$. Note that in the case of $\Z^2$, the set $S$ is the set of faces of the grid graph. A labeling $c : X \to [k]$ is then \emph{Borel $k$-polychromatic} if $c$ is Borel and every hypercube in $S$ contains all $k$ colors. The \emph{Borel polychromatic number} of $G$, denoted $\chi^p_B(G)$, is the greatest $k$ so that $G$ admits a Borel polychromatic $k$-coloring. Note that this differs from proper coloring: larger $k$ are more difficult to obtain than smaller ones. 

Requiring the labeling to be Borel fundamentally alters the landscape of coloring problems. In the definable setting of Borel combinatorics, global structural constraints emerge. In descriptive set theory, this phenomenon was first systematically studied by Kechris, Solecki, and Todorcevic \cite{KST1999}, who showed that the Borel chromatic number of Borel graphs can differ drastically from the classical setting. Since then, Borel combinatorics has flourished, with major contributions on measurable matchings \cite{ConleyTuckerDrob2015}, edge-colorings \cite{MarksUnger2015}, and connections to distributed algorithms and locally checkable labelings \cite{BBLW2025, Bernshteyn2020}. For surveys on the topic, see \cite{GrebikVidnyanszky, KechrisMarks}. These works illustrate the central theme that requiring measurability constraints creates new obstructions absent in the finite case.

Our main result fully characterizes the Borel polychromatic number of $\Z^d$.
 
\begin{theorem}
\label{thm:main}
Let $G$ be the grid graph induced by a free Borel action of $\mathbb{Z}^d$ on a standard Borel space. Then $G$ admits a Borel polychromatic $(2^d-1)$-coloring. 
\end{theorem}

The proof of the existence of a polychromatic
$(2^d-1)$ Borel coloring is given in Section 3. The proof uses a Borel
\emph{toast} construction, and Borel toasts for free Borel actions of $\Z^d$
which were first shown to exist by Gao, Jackson, Khrone, and Seward
\cite{GJKS2015}. We recall the definitions of these toasts in Section 2.

 Additionally, this bound is sharp;   although a polychromatic $2^d$-coloring exists in the classical setting, this coloring may not be Borel. Our proof of this first
establishes that any polychromatic $2^d$-coloring has to exhibit a
certain type of invariance in at least one generator. We then use a
standard ergodic-theoretic argument to show no such Borel coloring can
exist similarly to the proof that the irrational rotation of the
circle is not measurable $2$-colorable. In section 5, we discuss some
open problems about exactly how much rigidity must exist for
polychromatic $2^d$-colorings.

\section{Background}

We begin by recalling standard terminology from descriptive set theory and Borel combinatorics. Throughout, we write $X$ for a standard Borel space and $G$ for a Borel graph.

\subsection*{Borel graphs}

A \emph{standard Borel space} is a measurable space $(X,\mathcal{B})$ arising from the Borel $\sigma$-algebra $\mathcal{B}$ of a Polish topology on $X$. Equivalently, it is a set equipped with the $\sigma$-algebra generated by the open sets of some separable completely metrizable topological space. Such spaces are known as \textit{Polish}. A \emph{Borel graph} on $X$ is a Borel set $G \subseteq X \times X$ which is symmetric and disjoint from the diagonal. We view $G$ as the edges of the graph and $X$ as the vertex set. 

Given a graph $G$ on a standard Borel space $X$, and a countable set of labels $\Sigma$, a \emph{Borel labeling} of $G$ is a Borel function $c : X \to \Sigma$. Equivalently, a Borel labeling is a labeling where the set of points with a given label is a Borel set. We call a Borel labeling of $G$ a \textit{Borel proper k-coloring} if it is a proper coloring in the classical sense (no two adjacent vertices share a color) and the label set has at most $k$-many elements. The \emph{Borel chromatic number} $\chi_B(G)$ is the least $k$ for which there exists a Borel proper $k$-coloring of $G$. The Borel chromatic number is bounded below by the classical chromatic number of the graph, but may differ drastically. In particular, there are graphs which are classically two colorable, but are not even countably Borel colorable.

A Borel graph $G$ is \emph{hyperfinite} if it can be written as an increasing union of Borel subgraphs each of whose connected components are finite. Hyperfiniteness is central in Borel combinatorics: it captures the idea that a graph admits finite approximations in a Borel-uniform way. In particular, grids are hyperfinite.

\subsection*{Toast decompositions}

One of the most useful tools for working with hyperfinite Borel graphs is the notion of a \emph{toast} decomposition, introduced by Gao, Jackson, Krohne, and Seward \cite{GJS2009}. A toast provides a partition of a Borel graph into finite pieces that are ``nested'' in a controlled way, making it possible to build Borel solutions inductively. 

\begin{definition}
Let $G$ be a Borel graph on a standard Borel space $X$. A Borel $r$-\emph{toast} $\tau\subseteq [X]^{<\infty}$ for $G$ is a Borel collection of finite subsets of $X$, called \textit{toast pieces} so that: 
\begin{enumerate}
    \item Distinct $K,K'\in \tau$ cannot overlap nontrivially, and in fact, are $r$-apart in the graph metric: either $$B_r(K) \subseteq K', B_r(K') \subseteq K, \text{ or } B_r(K) \cap K' = \varnothing.$$ Here $r$-balls are taken in the graph metric, meaning that for $A\subseteq X$, $B_r(A)$ is defined as the set of vertices whose shortest $G$-path to a vertex in $A$ has length at most $r$. 
    \item We have that the pieces union to the whole space: $$\bigcup \tau = X.$$
\end{enumerate}
\end{definition}

Intuitively, a toast is a hierarchy of finite ``tiles'' that grow to cover each connected component of $G$.

Toast decompositions are particularly powerful for constructing Borel colorings and labelings: one colors the minimal toast pieces (those not containing any other pieces as a subset) at stage $0$, then extends the coloring piece by piece at higher stages while ensuring consistency. This inductive construction guarantees Borel definability, and is the main way toast is exploited in practice.

\begin{lemma}[Toast for grids, {\cite{GJKS2015}}]
\label{lem:toast}
Every Borel graph induced by a Borel action of $\Z^d$ admits a toast decomposition. 
\end{lemma}

\cref{lem:toast} leads to a proof of the fact that Borel grids are properly $3$-colorable.

\subsection*{Ergodicity}

Let $(X,\mu)$ be a standard probability space and $\Gamma \curvearrowright X$ a measure-preserving action of a countable group. The action is \emph{ergodic} if every $\Gamma$-invariant Borel set has measure $0$ or $1$. Ergodicity plays a key role in showing that certain Borel labelings cannot exist.

To illustrate that there are natural ergodic actions of $\mathbb{Z}^d$, it is helpful to recall the Bernoulli shift. 
More generally, for any $d \geq 1$, consider the Bernoulli shift action of $\mathbb{Z}^d$ on the product probability space 
\[
X = \{0,1\}^{\mathbb{Z}^d}, \qquad \mu = \left(\tfrac{1}{2},\tfrac{1}{2}\right)^{\otimes \mathbb{Z}^d}.
\]
For $g \in \mathbb{Z}^d$, the action is defined by 
\[
(g \cdot x)(h) = x(h-g) \qquad \text{for all } x \in X, \, h \in \mathbb{Z}^d.
\]
This is a measure-preserving action, since shifts of the index set permute the product measure.

Moreover, the Bernoulli shift is \emph{ergodic} \cite{KechrisMiller}: if $A \subseteq X$ is a measurable set with $g \cdot A = A$ for all $g \in \mathbb{Z}^d$, then by Kolmogorov’s zero--one law the membership of $x$ in $A$ must be determined by finitely many coordinates. 
But shift-invariance forces $A$ to have the same probability regardless of the values on those finitely many coordinates, hence $\mu(A)\in \{0,1\}$. 

We also have the stronger fact that this action of  $\mathbb{Z}^d$ is also ergodic when restricted to each of the $\mathbb{Z}$-subactions generated by the standard basis vectors. 
That is, shifting in the horizontal or vertical direction alone already yields an ergodic $\mathbb{Z}$-action. 
Thus the Bernoulli shift provides a canonical example of an action of $\mathbb{Z}^d$ that is ergodic in each coordinate component.

\section{Constructing $(2^d-1)$-colorings}
In this section, we show that every free Borel action of $\Z^d$ admits a $(2^d-1)$-polychromatic coloring.  
The key idea is that one can  construct highly structured colorings with $2^d-1$ colors by carefully interpolating between local cube configurations.  
We begin by establishing a combinatorial lemma (Lemma~\ref{cubecoloring}) ensuring that any two surjective labelings of the $d$-cube with $2^d-1$ colors can be connected through a sequence of labelings differing in only one vertex at a time.  Intuitively, we can slide between any two surjective cube colorings without ever losing surjectivity. This local flexibility will allow us to align repetitive colorings across distant regions, and ultimately, by means of a toast construction, extend them to a global Borel $(2^d-1)$-polychromatic coloring.

\begin{lemma}\label{cubecoloring}
Let $\ell$ be a set of $2^d-1$ colors.  
If $c_A,c_B:\{0,1\}^d\to\ell$ are surjective labelings of the $d$–cube,  
then there exists a sequence of $2^{d+1}$-many surjective labelings
\[
c_A=:c_0,c_1,\dots,c_{2^{d+1}}=:c_B
\]
such that each consecutive pair $c_t,c_{t+1}$ differs on at most one vertex of the cube.
\end{lemma}

\begin{proof}
We will transform $c_A$ into $c_B$ one vertex at a time, always preserving surjectivity.  
Enumerate the cube’s vertices as $v_0,\dots,v_{2^d-1}$ so that $v_{2^d-1}$ is a vertex whose color in $c_B$ is repeated.  
(Such a vertex exists since $c_B$ uses $2^d-1$ colors on $2^d$ vertices.)

\medskip\noindent
\emph{Inductive invariant.}  
We construct surjective labelings $c_{2i}$ for $i=0,1,\dots,2^d$ so that
\[
c_{2i}(v_j)=c_B(v_j)\quad\text{for all }j<i.
\]
Thus after $2i$ steps, the first $i$ vertices have already been corrected. Let $c_0=c_A$. 

\medskip\noindent
\emph{Inductive step.}  
Suppose $c_{2i}$ has been defined with the invariant above.  
Since $c_{2i}$ is surjective onto $\ell$, it must use some color twice.  
Let $k\ge i$ be the largest index such that $c_{2i}(v_k)$ is a repeated color.  
(Note: $k\ge i$ because among $v_0,\dots,v_{i-1}$ the colors are already fixed to $c_B$, which are all distinct except at $v_{2^d-1}$.)

Now define the next two labelings:
\[
c_{2i+1}(v_j)=
\begin{cases}
c_{2i}(v_i), & j=k,\\
c_{2i}(v_j), & j\neq k,
\end{cases}
\qquad
c_{2i+2}(v_j)=
\begin{cases}
c_B(v_i), & j=i,\\
c_{2i+1}(v_j), & j\neq i.
\end{cases}
\]

\medskip\noindent
Note that:
\begin{itemize}
\item Each step changes at most one vertex.  
\item $c_{2i+1}$ remains surjective: it differs from $c_{2i}$ only at $v_k$, which had a repeated color in $c_{2i}$.  
Thus all $|\ell| = 2^d-1$ colors still appear.  
\item In $c_{2i+1}$ the repeated color is now $c_{2i+1}(v_i)$, so changing $v_i$ in $c_{2i+2}$ also preserves surjectivity.  
\item Finally, $c_{2i+2}(v_i)=c_B(v_i)$ and for $j<i$ we still have $c_{2i+2}(v_j)=c_B(v_j)$ by the inductive hypothesis.  
Hence the invariant is advanced from $i$ to $i+1$.
\end{itemize}

\medskip
By induction this process terminates at $c_{2^{d+1}}=c_B$.  
Thus we obtain the desired sequence of surjective labelings.
\end{proof}

We now define a specific type of polychromatic coloring, called a \textit{repetitive} coloring, that we will use as the building blocks along the toast when constructing our final Borel coloring. 

\begin{definition}
     Let $G$ be a Borel graph induced by a free Borel action of $\Z^d = \langle e_0,\dots, e_{d-1}\rangle$. We say that a labeling of the vertices of $G$ is \emph{repetitive} if for every $i<d$, the labeling is $e_i^2$-invariant. Note that equivalently, these are the labelings which factor through to a quotient $\Z_2^d=(\Z/2\Z)^d$ of $\Z^d$ on each individual orbit. 
\end{definition}

Note that given a partial repetitive coloring of a region containing a cube, there is a unique repetitive extension of the coloring to $\Z^d$. This is because, since such colorings factor through $\Z_2^d$, they are exactly defined by their values on a cube.  

Given a set $K\subseteq\Z^d$ and $r\in\N$, we let $B_r(K)\subseteq\Z^d$ be the set of $x\in \Z^d$ whose distance to $K$ in the graph metric is at most $r$.

\begin{theorem}\label{thm:2.3}
    Let $G$ be a grid graph induced by a free Borel action of $\Z^d$. Then $G$ admits a Borel polychromatic coloring with $2^d-1$ colors.
\end{theorem}

 \begin{proof}[Proof of Theorem~\ref{thm:2.3}]
Fix $d\in\mathbb{N}$ and a free Borel action $\mathbb{Z}^d=\langle e_0,\dots,e_{d-1}\rangle\actson X$ with grid graph $G$.
Let $\ell$ be a color set with $|\ell|=2^d-1$ and 
choose any surjective labeling $a:\Z_2^d\to\ell$. 

\smallskip\noindent\emph{Repetitive templates.}
For any finite $K\in[X]^{<\infty}$, we can choose a ``root'' $r_K\in K$ in a Borel way. Given this root, we can define a map $\varphi_K: K \to \Z_2^d$ by taking entries mod $2$ starting at $r_K$, meaning that $\varphi_K(g\cdot r_K)=g\bmod 2$ for $g\in\Z^d$. 

The \emph{repetitive template} on $K$ is the labeling $c^K:K\to 
\ell $ given by starting at $r_K$ and coloring according to $a$ mod $2$. That is, we have $$c^K = a\circ \varphi_K .$$
Since $a$ is surjective, $c^{K}$ is a repetitive $(2^d\!-\!1)$-polychromatic coloring on $K$.

\smallskip\noindent\emph{Numerical parameters.}
By Lemma~\ref{cubecoloring}, any two surjective labelings of $Q$ can be joined by a sequence of length
at most $2^{d+1}$ with single-vertex changes. Set
\[
R:= 2^{d+2}\cdot d,\qquad r:= (2^{d+3}+1)\cdot d.
\]
Take a Borel $r$-toast $\tau$ for $G$ (Lemma~\ref{lem:toast}). The choice $r= 2R+d$ ensures:
for any toast pieces $L\neq L'$ with $B_r(L),B_r(L')\subseteq K$ we have
$\mathrm{dist}(L,L')\ge r$, hence the $R$-thickenings $B_R(L),B_R(L')$ are disjoint and moreover more than $d$ apart;
also $B_R(L)\subseteq K$ (since $r>R$).

\smallskip\noindent\emph{Induction over toast rank.}
Partially order $\tau$ by inclusion; since pieces are finite and the $r$-toast is nested/disjoint, this is well-founded.
We construct Borel colorings $c\!\restriction K$ for all $K\in\tau$ by induction so that:

\begin{itemize}
\item[(I1)] $c\!\restriction K$ is $(2^d\!-\!1)$-polychromatic on $K$;
\item[(I2)] if $L\subseteq K$ is an internal toast piece, then $c\!\restriction L$ is already defined and
$c\!\restriction L$ is preserved when extending to $K$;
\item[(I3)] every unit cube $C\subseteq K$ meets either the ``exterior'' region (where $c=c^{K}$) or exactly one
``gap'' around a single internal piece $L\subseteq K$ and, inside that gap, $C$ intersects at most two consecutive shells. This will be defined below. 
\end{itemize}

\smallskip\noindent\emph{Basis.}
If $K$ is a $\subseteq$–minimal toast piece in $\tau$, set $c\!\restriction K:=c^{K}$. Then (I1)--(I3) are immediate.

\smallskip\noindent\emph{Inductive step.}
Fix $K\in\tau$ and list its maximal internal pieces as $L_1,\dots,L_m$.
By induction, $c$ is defined on $L_i$ for each $i$.
Work in the embedded copy $\varphi_K(K)\subseteq\mathbb{Z}^d$.

\medskip\noindent
\textbf{Step A: exterior.}
Define the exterior region
\[
E_K:=K\setminus \bigcup_{i=1}^m B_R(L_i).
\]
Color $E_K$ by the repetitive template: $c(x):=c^{K}(x)$ for $x\in E_K$.

\medskip\noindent
\textbf{Step B: disjoint gaps.}
For each $i$, define the ``gap'' around $L_i$, denoted $P_i$ by:
\[
P_i:=B_R(L_i)\setminus L_i.
\]
By the choice of $r$, the $P_i$ are pairwise disjoint and contained in $K$. Moreover, each unit hypercube $C=\{0,1\}^d\cdot x$ in the same $G$-connected component $[K]_G$ as $K$ can intersect at most one of $B_R(L_1),\ldots,B_R(L_m)$, and $[K]_G\setminus K$. 
On each gap $P_i$ independently, we connect the coloring of $L_i$ with the coloring $c_K$ on the exterior.

\medskip\noindent
\textbf{Step C: connecting the gaps.}
Fix $i$.
Consider the two repetitive polychromatic colorings on $K$:
the \emph{inner} repetitive template $c^{L_i}$ we have from when we defined the labeling on $L_i$,  and the \emph{outer} template $c^{K}$.
We can view these as functions on $\Z_2^d$, since they are repetitive. Hence by Lemma~\ref{cubecoloring} there exists a sequence of surjective labelings of $\Z_2^d$, 
\[
c^{L_i}\circ\varphi_K^{-1}=a_0,\,a_1,\,\dots,\,a_{2^{d+1}}=a=c^K\circ\varphi_K^{-1}
\]
with at most single-vertex changes between each step.
We can extend each $a_t$ to a repetitive coloring on $K$ via $\varphi_K$. 

Partition $P_i$ into shells of width $2d$ around $L_i$:
\[
S_{i,t}:=\big\{x\in P_i:\ 2dt\le \mathrm{dist}(x,L_i)<2d(t+1)\big\},\qquad t=0,1,\dots,2^{d+1}.
\]

Define $c(x)$ by coloring the $t$-th shell by $a_t$: $c(x):=a_t\circ \varphi_K(x)$ for $x\in S_{i,t}$.
This fixes $c$ on $P_i$ and, together with Step A and the inductive coloring on $L_i$, defines $c$ on $K$.

\smallskip\noindent\emph{Verification.}
Because the graph metric distance changes by at most $d$ across unit hypercubes, every unit cube $C\subseteq K$ is contained either in $E_K$
or in $S_{i,t}\cup S_{i,t+1}$ for a \emph{unique} $i$ and some $t$, by disjointness of the $P_i$. 
On $E_K$, $c=c^{K}$ is $(2^d\!-\!1)$-polychromatic.
On $S_{i,t}\cup S_{i,t+1}$ we have $\tilde c_t=\varphi_K\circ a_t$ and $\tilde c_{t+1}=\varphi_K \circ a_{t+1}$; by construction the two differ on at most one
vertex of any unit cube, so $c\!\restriction C$ equals either $\tilde c_t\!\restriction C$ or $\tilde c_{t+1}\!\restriction C$,
both of which are surjective onto $\ell$.
Thus (I1)--(I3) hold for $K$, and $c$ extends the inductive coloring on the $L_i$'s, proving (I2).

\smallskip
Since the construction is Borel at each stage and proceeds by induction over a Borel well-founded order on $\tau$,
the resulting global $c:X\to\ell$ is Borel.
Finally, every unit cube of $G$ is contained in some toast piece, so by the verification above $c$ is
$(2^d\!-\!1)$-polychromatic on $G$.
\end{proof}

In the next section we establish obstructions to achieving a full $2^d$-polychromatic coloring, thus showing that \cref{thm:2.3} is best possible.  

\section{Characterizing $2^d$-colorings}

In this section we analyze when a Borel $2^d$-polychromatic coloring of a grid graph can exist. 
The guiding principle is that such a coloring is extremely rigid: every unit hypercube must contain all $2^d$ colors, which forces each hypercube to be colored injectively. 
This constraint links $2^d$-polychromatic colorings to families of Borel $2$-colorings of the individual $\Z$-subactions of $\Z^d$. 
We show that the existence of a $2^d$-polychromatic coloring implies the existence of compatible Borel $2$-colorings along each $\Z$-line, and then we investigate to what extent the converse holds. 
This leads to a precise characterization of $2^d$-colorings.

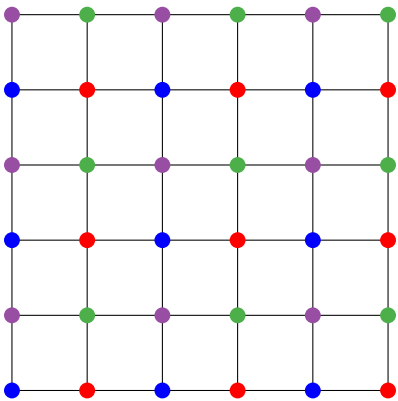
\begin{figure}[h]\label{fig:4polygrid}
\centering
\begin{tikzpicture}[scale=1]
  \foreach \i in {0,...,5} {
    \draw (0,\i) -- (5,\i);
    \draw (\i,0) -- (\i,5);
  }

  \definecolor{cPurple}{RGB}{152,78,163}
  \definecolor{cGreen}{RGB}{77,175,74}

  \foreach \y in {0,...,5}{
    \pgfmathtruncatemacro{\rowParity}{mod(\y,2)} 
    \foreach \x in {0,...,5}{
      \pgfmathtruncatemacro{\colParity}{mod(\x,2)} 
      \ifnum\rowParity=0
        \ifnum\colParity=0
          \fill[blue] (\x,\y) circle (3pt);
        \else
          \fill[red] (\x,\y) circle (3pt);
        \fi
      \else
        \ifnum\colParity=0
          \fill[cPurple] (\x,\y) circle (3pt);
        \else
          \fill[cGreen] (\x,\y) circle (3pt);
        \fi
      \fi
    }
  }
\end{tikzpicture}
\caption{A $4$-polychromatic coloring of a segment of the grid.}
\end{figure}

\subsection*{Example: Classical polychromatic number of the grid}

Before turning to the Borel setting, it is helpful to recall the classical case. 
Consider the infinite planar grid graph, and let $\mathcal{S}$ be the collection of its unit squares (faces). 
A coloring $c:V(G)\to[k]$ is $k$-polychromatic if every unit square contains all $k$ colors. 
It is easy to see that $\chi^p(G,\mathcal{S})=4$: each unit square has four vertices, so at most four colors can appear, and this bound is sharp since one can color the grid so that every square contains all four colors, as shown in \cref{Ex:classical-grid}.

\begin{example}\label{Ex:classical-grid}
Figure 1 shows a $4$-polychromatic coloring of a finite segment of the grid graph. 
Every unit square (face) contains each of the four colors.
\end{example}

In this example, we have a $2$-coloring of the Schreier graphs of the $\Z$ sub-actions defined by the generators. We will see in the next section, that given a $4$-polychromatic coloring, this will always be the case.

\subsection*{Impossibility of Borel $2^d$-polychromatic colorings }
First, it is helpful to understand how obstructions to polychromatic colorings arise.  
In particular, the existence of a $2^d$-polychromatic coloring on a $\Z^d$-grid is tied to the ability to Borel $2$-color the one-dimensional $\Z$-subactions.  In fact, if a single generator of the action already fails to admit a Borel proper $2$-coloring, as in the classical example of an irrational rotation of the circle, then this is already an obstruction to producing a $2^d$-polychromatic coloring.  
The following lemma makes this precise by showing that any obstruction to $2$-colorability of a $\Z$-subaction propagates to obstruct the existence of a Borel $2^d$-polychromatic coloring on the whole grid.

\begin{lemma}\label{lem:no-poly}
Let $\mathbb{Z}^d\actson X$ be a free Borel action on a standard Borel space with associated grid graph $G$.  
If at least one generator induces a free Borel $\mathbb{Z}$–subaction of $X$ without a Borel proper $2$-coloring,  
then $G$ does not admit a Borel $2^d$-polychromatic coloring.
\end{lemma}

For instance, if $\mathbb{Z}^d$ acts on the torus $\mathbb{R}/\mathbb{Z}$ by $d$-many $\mathbb{Q}$–linearly independent irrational rotations, then no Borel $2^d$-polychromatic coloring exists.

\begin{proof}
We will prove the contrapositive: We will show that given a $2^d$-polychromatic coloring, we get a $2$-coloring of  any $\Z$-subaction by a generator of $\Z^d$.  Write $\mathbb{Z}^d=\langle e_0,\dots,e_{d-1}\rangle$. Without loss of generality, we will show this for the action of $e_0$. 
Consider the unit $d$–cube
\[
C:=\{0,1\}^d\subseteq\mathbb{Z}^d,
\]
and let $C_0\subseteq C$ be the $(d-1)$–dimensional face obtained by setting the $e_0$–coordinate to be $0$.  
Then
\[
C = C_0 \ \sqcup\ e_0\cdot C_0,
\]
so each $d$–cube naturally splits into two parallel $(d-1)$–faces.

\medskip
Suppose that $c:X\to\ell$ is a Borel $2^d$-polychromatic labeling, with $|\ell|=2^d$.  
For each $x\in X$, the cube $C\cdot x$ has $2^d$ vertices and must see all $2^d$ colors.  
Thus $c$ is injective on $C\cdot x$, and in particular
\[
c(C_0\cdot x)\ \sqcup\ c(e_0 C_0\cdot x) \ =\ \ell.
\]
Now compare the two adjacent cubes $C\cdot x$ and $C\cdot (e_0\cdot x)$.  
Both contain the face $e_0 C_0\cdot x$, so the above identity implies
\[
c(C_0\cdot x)=\ell\setminus c(e_0 C_0\cdot x)=c(e_0^2 C_0\cdot x).
\]

\medskip
Fix a color $i\in\ell$, and define a map $c':X\to\{0,1\}$ by
\[
c'(x)=
\begin{cases}
1 &\text{if } i\in c(C_0\cdot x),\\
0 &\text{otherwise}.
\end{cases}
\]
This function is Borel, since membership in $c(C_0\cdot x)$ is a Borel property.

\medskip
Finally, observe that $c'$ is a proper $2$-coloring of the graph induced by the $e_0$–action:  
if $x$ and $e_0\cdot x$ are adjacent, then
\[
c'(e_0\cdot x)=1 \iff i\in c(C_0\cdot e_0 x)
= c(e_0^2 C_0\cdot x),
\]
which is the complement of $c(C_0\cdot x)$ with respect to $\ell$.  
Hence $c'(e_0\cdot x)\neq c'(x)$ for all $x$.  
This gives a Borel proper $2$-coloring of the $e_0$-subgraph.
\end{proof}
 
The obstruction in Lemma~\ref{lem:no-poly} immediately yields the following corollary. 

\begin{cor}\label{cor:1.1}
If $\mathbb{Z}^d=\langle e_0,\ldots,e_{d-1}\rangle\actson X$ is a free Borel action whose grid graph $G$ admits a Borel $2^d$-polychromatic coloring, then for each $i< d$ there is a Borel proper $2$-coloring $c_i$ of the $i$-th $\mathbb{Z}$-subaction $\langle e_i\rangle\actson X$.
\end{cor}

Corollary~\ref{cor:1.1} says that the existence of a $2^d$-polychromatic coloring automatically produces compatible $2$-colorings along each one-dimensional direction.  
The next question is: how much extra symmetry or invariance do these $2$-colorings have, and when is that enough to recover a full $2^d$-polychromatic coloring?  
To capture this, we introduce two notions.

\begin{definition}
Let $\mathbb{Z}^d=\langle e_0,\ldots,e_{d-1}\rangle\actson X$ be a free action.  
A proper $2$-coloring $c_i:X\to\{0,1\}$ of the $i$th $\mathbb{Z}$-subaction is called \emph{orthogonally invariant} in direction $i$ if it is unchanged under the other generators, i.e.
\[
c_i(e_j\cdot x)=c_i(x)\quad\text{for all }x\in X,\ j\neq i.
\]
\end{definition}

Figure~2 illustrates an orthogonally invariant $2$-coloring for the vertical direction in $\mathbb{Z}^2$.

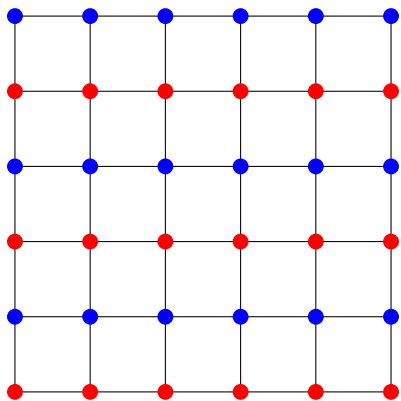
\begin{figure}[h]\label{fig:ortho}
\centering
\begin{tikzpicture}[scale=1]
  \foreach \i in {0,...,5} {
    \draw (0,\i) -- (5,\i);
    \draw (\i,0) -- (\i,5);
  }

  \definecolor{cPurple}{RGB}{152,78,163}
  \definecolor{cGreen}{RGB}{77,175,74}

  \foreach \y in {0,...,5}{
    \pgfmathtruncatemacro{\rowParity}{mod(\y,2)} 
    \foreach \x in {0,...,5}{
        \ifnum\rowParity=0
          \fill[red] (\x,\y) circle (3pt);
        \else
          \fill[blue] (\x,\y) circle (3pt);
      \fi
    }
  }
\end{tikzpicture}

\caption{An orthogonally invariant proper $2$-coloring for the vertical $\Z$-subaction. }
\end{figure}

\begin{definition}
Let $(c_0,\ldots,c_{d-1})$ be a $d$-tuple of proper $2$-colorings, one for each $\mathbb{Z}$-subaction $\langle e_i\rangle\actson X$.  
We say that this tuple is \emph{$n$-fold invariant} if on every orbit $[x]_{\mathbb{Z}^d}$, there are at least $n$-many indices $i$ such that $c_i$ is orthogonally invariant on $[x]_{\mathbb{Z}^d}$.  
\end{definition}

Thus $0$-fold invariance is the minimal structure guaranteed by Corollary~\ref{cor:1.1}, while $(d-1)$-fold invariance says that in all the other coordinate directions except possibly one, the coloring is orthogonally invariant.

With this definition in hand, our previous Corollary \ref{cor:1.1} now says that if $G$ is a grid graph of a free Borel $\Z^d$-action, then the existence of a Borel $2^d$-polychromatic coloring implies the existence of a $0$-fold invariant $d$-tuple of Borel $2$-colorings. We collect and then strengthen this observation in the theorem below. 

\begin{theorem}\label{thm:1.4}
    Fix a free Borel action $\Z^d=\langle e_0,\ldots,e_{d-1}\rangle\actson X$ and its corresponding grid graph $G$. The following hold. 
    \begin{enumerate}
        \item[(0)] If $G$ admits a Borel $2^d$-polychromatic coloring, then $G$ admits a $0$-fold invariant $d$-tuple of Borel $2$-colorings.
        \item[(1)] If $G$ admits a $(d-1)$-fold invariant $d$-tuple of Borel $2$-colorings, then $G$ admits a Borel $2^d$-polychromatic coloring. 
        \item[(2)]
        When $d=2$, $G$ admits a Borel $4$-polychromatic coloring if and only if $G$ admits a $1$-fold invariant $2$-tuple of Borel $2$-colorings. 
    \end{enumerate}
\end{theorem}

\begin{proof}
    Part (0) is exactly Corollary~\ref{cor:1.1}.

For (1), suppose $(c_0,\ldots,c_{d-1})$ is a $(d-1)$-fold invariant $d$-tuple of $2$-colorings.  
Define
\[
c:X\to\mathbb{Z}_2^d,\qquad c(x)=(c_0(x),\ldots,c_{d-1}(x)).
\]
Because all but possibly one of the coordinates are orthogonally invariant, each cube $\{0,1\}^d\cdot x$ sees all $2^d$ distinct values of $c$, so $c$ is a Borel $2^d$-polychromatic coloring.

For (2), one direction is already covered: a $1$-fold invariant pair of $2$-colorings yields a $4$-polychromatic coloring by the same construction as (1).  
Conversely, assume $c:X\to\{0,1,2,3\}$ is a $4$-polychromatic coloring for a $\mathbb{Z}^2$–action with generators $e_0,e_1$.  
By Corollary~\ref{cor:1.1}, we can extract two associated Borel $2$-colorings $c_0,c_1$. We will check that these are $1$-fold invariant. 

    Assume $c:X\to\{0,1,2,3\}$ is a Borel $4$-polychromatic coloring and let $\Z^2=\langle e_0,e_1\rangle$. Recall that we defined the $2$-tuple of Borel $2$-colorings $c_0,c_1:X\to\{0,1\}$ such that for $i\in\{0,1\}$ and $x\in X$, $c_i(x)=1$ if and only if $0\in\{c(x),c(e_{1-i}x)\}$.

    We claim that for every $x\in X$, the coloring $c\restriction[x]_{\Z^2}$ is either $e_0^2$-invariant or $e_1^2$-invariant. Once we know this, it is then easy to see that if $c\restriction[x]_{\Z^2}$ is $e_i^2$-invariant for some $i\in\{0,1\}$, then $c_{1-i}$ is $e_i$-invariant, and we witness the conclusion that $(c_0,c_1)$ is $1$-fold invariant. 

    From now on, we fix some $x\in X$ such that $c\restriction[x]_{\Z^2}$ is neither $e_0^2$-invariant nor $e_1^2$-invariant and work towards a contradiction. Note that for $i\in\{0,1\}$, since $c$ itself is a proper coloring without monochromatic $e_i$-edges, if for some $y\in X$, the image $c(\langle e_i\rangle\cdot y)$ contains only $2$ out of the $4$ possible colors $\{0,1,2,3\}$, then $c\restriction\langle e_i\rangle\cdot y$ is $e_i^2$-invariant. Therefore, since $c\restriction[x]_{\Z^2}$ is not $e_i^2$-invariant, there are elements $y_0,y_1\in[x]_{\Z^2}$ such that $c([y_0]_{\langle e_0\rangle}),c([y_1]_{\langle e_1\rangle})\subseteq\{0,1,2,3\}$ both contain $3$ or more colors. 

    For $i\in\{0,1\}$, if for all $n\in\Z$, $\{c(e_i^{n}y_i),c(e_i^{n+1}y_i),c(e_i^{n+2}y_i)\}$ contains only $2$ out of the $4$ colors $\{0,1,2,3\}$, then since $c$ is a proper coloring without monochromatic $e_i$-edges, these triplets must all contain the same $2$ colors, and $c([y_i]_{\langle e_i\rangle})$ contains only $2$ colors, which contradicts our assumption about $y_i$. Therefore there exists some $n=n_i\in\Z$ for which $\{c(e_i^{n}y_i),c(e_i^{n+1}y_i),c(e_i^{n+2}y_i)\}$ contains exactly $3$ colors, and we may assume without loss of generality that $\{c(y_i),c(e_iy_i),c(e_i^2y_i)\}$ contains $3$ colors for $i=\{0,1\}$. 

\begin{figure}[h]\label{fig:proofpic}
\centering
\begin{subcaptionblock}{0.3\textwidth} 
  \centering
  \begin{tikzpicture}
    \useasboundingbox (-1,-.7) rectangle (3,1.7);
    \begin{scope}
      \draw (0,0) -- (1,0);
      \draw (1,0) -- (2,0);
      \draw (0,0) -- (0,1);
      \draw (1,0) -- (1,1);
      \draw (2,0) -- (2,1);
      \draw (0,1) -- (1,1);
      \draw (1,1) -- (2,1);
    \end{scope}
    \definecolor{cGreen}{RGB}{77,175,74}
    \fill[red] (0,0) circle (3pt);
    \fill[cGreen] (1,0) circle (3pt);
    \fill[blue] (2,0) circle (3pt);
    \draw[fill=white] (0,1) circle (3pt);
    \draw[fill=white] (1,1) circle (3pt);
    \draw[fill=white] (2,1) circle (3pt);

    \begin{scope}[text height=1.5ex,text depth=1ex,every node/.style={scale=.8,inner xsep=5pt,inner ysep=8pt}]
    \node at (0,0) [anchor=north east] {$e_{1-i}^ny_i$};  
    \node at (1,0) [anchor=north] {$e_{1-i}^ne_iy_i$};  
    \node at (2,0) [anchor=north west] {$e_{1-i}^ne_i^2y_i$};  
    \node at (0,1) [anchor=south east] {$e_{1-i}^{n+1}y_i$};  
    \node at (1,1) [anchor=south] {$e_{1-i}^{n+1}e_iy_i$};  
    \node at (2,1) [anchor=south west] {$e_{1-i}^{n+1}e_i^2y_i$};  
    \end{scope}
  \end{tikzpicture}
  \caption{The inductive hypothesis.}
\end{subcaptionblock}
~
\begin{subcaptionblock}{0.3\textwidth} 
  \centering
  \begin{tikzpicture}
   \definecolor{cPurple}{RGB}{152,78,163}
  \definecolor{cGreen}{RGB}{77,175,74}
    \useasboundingbox (-1,-.7) rectangle (3,1.7);
    \begin{scope}
      \draw (0,0) -- (1,0);
      \draw (1,0) -- (2,0);
      \draw (0,0) -- (0,1);
      \draw (1,0) -- (1,1);
      \draw (2,0) -- (2,1);
      \draw (0,1) -- (1,1);
      \draw (1,1) -- (2,1);
    \end{scope}
    \fill[red] (0,0) circle (3pt);
    \fill[cGreen] (1,0) circle (3pt);
    \fill[blue] (2,0) circle (3pt);
    \draw[fill=white] (0,1) circle (3pt);
    \fill[cPurple] (1,1) circle (3pt);
    \draw[fill=white] (2,1) circle (3pt);

    \begin{scope}[->,shorten <=7pt,shorten >=7pt]
      \draw (0,0) -- (1,1);
      \draw (2,0) -- (1,1);
    \end{scope}

  \end{tikzpicture}
  \caption{Determining $c(e_{1-i}^{n+1}e_iy_i)$.}
\end{subcaptionblock}

\begin{subcaptionblock}{0.3\textwidth} 
  \centering
  \begin{tikzpicture}
   \definecolor{cPurple}{RGB}{152,78,163}
  \definecolor{cGreen}{RGB}{77,175,74}
    \useasboundingbox (-1,-.7) rectangle (3,3.5);
    \begin{scope}
      \draw (0,0) -- (1,0);
      \draw (1,0) -- (2,0);
      \draw (0,0) -- (0,1);
      \draw (1,0) -- (1,1);
      \draw (2,0) -- (2,1);
      \draw (0,1) -- (1,1);
      \draw (1,1) -- (2,1);
    \end{scope}
    \fill[red] (0,0) circle (3pt);
    \fill[cGreen] (1,0) circle (3pt);
    \fill[blue] (2,0) circle (3pt);
    \fill[blue] (0,1) circle (3pt);
    \fill[cPurple] (1,1) circle (3pt);
    \fill[red] (2,1) circle (3pt);
  \end{tikzpicture}
  \caption{Determining the rest.}
\end{subcaptionblock}
~
\begin{subcaptionblock}{0.3\textwidth} 
  \centering
  \begin{tikzpicture}
   \definecolor{cPurple}{RGB}{152,78,163}
  \definecolor{cGreen}{RGB}{77,175,74}
    \useasboundingbox (-1,-.7) rectangle (3,3.5);
    \begin{scope}
      \draw (0,0) -- (2,0);
      \draw (0,1) -- (2,1);
      \draw (0,2) -- (2,2);
      \draw (0,-.5) -- (0,2.5);
      \draw (1,-.5) -- (1,2.5);
      \draw (2,-.5) -- (2,2.5);
    \end{scope}
    \fill[red] (0,0) circle (3pt);
    \fill[cGreen] (1,0) circle (3pt);
    \fill[blue] (2,0) circle (3pt);
    \fill[blue] (0,1) circle (3pt);
    \fill[cPurple] (1,1) circle (3pt);
    \fill[red] (2,1) circle (3pt);
    \fill[red] (0,2) circle (3pt);
    \fill[cGreen] (1,2) circle (3pt);
    \fill[blue] (2,2) circle (3pt);

    \node at (0,3) {$\vdots$};
    \node at (1,3) {$\vdots$};
    \node at (2,3) {$\vdots$};
  \end{tikzpicture}
  \caption{The induction continues.}
\end{subcaptionblock}
\caption{An illustration of the proof of Theorem \ref{thm:1.4}.}
\end{figure}

    We next claim that for all $n\in\Z$, 
    \[c(e_{1-i}^{n+1}y_i)=c(e_{1-i}^ne_i^2y_i)\ne c(e_{1-i}^ny_i)=c(e_{1-i}^{n+1}e_i^2y_i),\]
    and we prove this by forward-backward induction on $n\in\Z$. We know that $c(y_i)\ne c(e_i^2y_i)$ which is the case $n=0$. For the inductive step, we assume $c(e_{1-i}^ny_i)\ne c(e_{1-i}^ne_i^2y_i)$. Since $c$ still has no monochromatic $e_i$-edges, $\{c(e_{1-i}^ny_i),c(e_{1-i}^ne_iy_i),c(e_{1-i}^ne_i^2y_i)\}$ contains $3$ colors. Consider the color $c(e_{1-i}^{n+1}e_iy_i)$: Since $c$ also has no monochromatic $e_{1-i}^{-1}e_i^{-1}$, $e_{1-i}^{-1}$, or $e_{1-i}^{-1}e_i$-edges, $c(e_{1-i}^{n+1}e_iy_i)$ must be distinct from $c(e_{1-i}^ny_i),c(e_{1-i}^ne_iy_i),c(e_{1-i}^ne_i^2y_i)$ and it is the missing fourth color. Therefore we know that
    \[\{c(e_{1-i}^{n+1}y_i)\}=\{0,1,2,3\}\setminus\{c(e_{1-i}^{n+1}e_iy_i),c(e_{1-i}^{n}y_i),c(e_{1-i}^{n}e_iy_i)\}=\{c(e_{1-i}^ne_i^2y_i)\}.\]
    The other cases, $c(e_{1-i}^ny_i)=c(e_{1-i}^{n+1}e_i^2y_i)$ and the backward induction case, are similar. This concludes the induction on $n\in\Z$. 

    Finally, since $y_0,y_1$ lie in the same orbit $[x]_{\Z^2}$, there are integers $m,n\in\Z$ such that $e_1^my_0=e_0^ny_1$. By the previous claim, we have $c(e_1^my_0)\ne c(e_1^me_0^2y_0)$, and yet $c(e_0^ny_1)=c(e_0^{n+1}e_1^2y_1)=c(e_0^{n+2}y_1)$, which contradicts $e_1^my_0=e_0^ny_1$. 
\end{proof}

\section{Open Questions}

Our main result establishes that the Borel polychromatic number of a free $\mathbb{Z}^d$-action is exactly $2^d-1$. 
This completely resolves the case of grid actions, but a number of natural directions remain open. 
We highlight several questions that suggest a broader theory of polychromatic colorings within descriptive combinatorics.

\medskip

\noindent\textbf{Beyond $\mathbb{Z}^d$.}  
The natural next step is to ask about groups other than $\mathbb{Z}^d$. 
In particular, what is the correct threshold for groups whose Cayley graphs share geometric or combinatorial features with grids?

\begin{ques}
    Given a group $G$ with planar Cayley graph, what can be said about the Borel polychromatic number for free Borel actions of $G$?
\end{ques}

\medskip

\noindent\textbf{Changing the underlying shapes.}  
Our definition of polychromatic colorings requires every unit cube in $\mathbb{Z}^d$ to see all colors. 
It is natural to replace cubes with more general tiling or face sets, for instance by allowing overlapping translates of an arbitrary finite pattern. 
This leads to the following generalization:

\begin{ques}
    Let $G$ be a grid graph of $\mathbb{Z}^d$, let $T \subseteq \mathbb{Z}^d$ be a finite shape, and let $S := \mathbb{Z}^d T$ be its set of translates. 
    What can we say about $\chi^p_B(G)$ when our family of sets is taken to be $S$?
\end{ques}

A broader challenge is to characterize which families of shapes $T$ admit the same threshold $|T|-1$, and which lead to fundamentally different behavior.

\medskip

\noindent\textbf{Rigidity vs.\ flexibility of $2^d$-colorings.}  
Finally, we ask whether Borel $2^d$-polychromatic colorings encode strictly stronger structure than tuples of $2$-colorings of the one-dimensional subactions.

\begin{ques}\label{que:1.5}
    Let $\mathbb{Z}^d \curvearrowright X$ be a free Borel action and let $G$ be its grid graph. 
    When $d \geq 3$, is the existence of a Borel $2^d$-polychromatic coloring strictly stronger than the existence of a ($0$-fold invariant) $d$-tuple of Borel $2$-colorings of the coordinate subactions?
\end{ques}

Question~\ref{que:1.5} can be viewed as asking about the rigidity or flexibility of general $2^d$-polychromatic colorings on $\mathbb{Z}^d$. 
A positive answer (the ``rigid case'') might be approached using notions of ergodicity, such as the grid-periodicity phenomena studied in \cite{GJKS2015}, combined with a sharper structural analysis of all $2^d$-colorings, similar to our treatment in Theorem~\ref{thm:1.4} where every $4$-polychromatic coloring of the $\mathbb{Z}^2$-grid was shown to be invariant under either $e_0^2$ or $e_1^2$. 
On the other hand, a negative answer (the ``flexible case'') may be approachable via toast constructions, as in Theorem~\ref{thm:2.3}. 

\medskip

\noindent We hope that progress on these problems will lead to a systematic theory of polychromatic colorings in the Borel setting, paralleling the rich development of Borel chromatic numbers over the last two decades.

\section*{Acknowledgements}

The authors would like to thank Andrew Marks and Robin Tucker-Drob for their guidance, support, and many helpful conversations throughout the development of this work.  
The first author is supported by the National Science Foundation Graduate Research Fellowship under Grant No.~DGE-2146752 and by a Simons Foundation Dissertation Fellowship in Mathematics.

\bibliography{mybib}{}
\bibliographystyle{plain}
\end{document}